\documentclass[11 pt]{amsart}
\usepackage{amscd,amsfonts,amssymb,amsmath}
\usepackage{hyperref}
\usepackage{epsfig}
\newtheorem{theorem}{Theorem}[section]
\newtheorem{corollary}[theorem]{Corollary}

\newtheorem{proposition}[theorem]{Proposition}
\theoremstyle{definition}

\newtheorem{problem}[theorem]{Problem}
\newtheorem{definition}[theorem]{Definition}
\newtheorem{example}[theorem]{Example}
\newtheorem{remark}[theorem]{Remark}
\numberwithin{equation}{subsection}
\newtheorem*{ack}{Acknowledgement}
\usepackage[all,cmtip]{xy}
\newcommand{\Alex}{\operatorname{Alex}}
\newcommand{\Aut}{\operatorname{Aut}}

\newcommand{\Conj}{\operatorname{Conj}}

\newcommand{\Core}{\operatorname{Core}}

\newcommand{\C}{\operatorname{C}}

\newcommand{\Inn}{\operatorname{Inn}}

\newcommand{\Sa}{\operatorname{S}}

\newcommand{\Hom}{\operatorname{Hom}}

\newcommand{\id}{\mathrm{id}}
\newcommand{\Fix}{\mathrm{Fix}}

\usepackage[autostyle]{csquotes}

\begin{document}
\title{Free quandles and knot quandles are residually finite}

\author{Valeriy G. Bardakov}
\author{Mahender Singh}
\author{Manpreet Singh}

\address{Sobolev Institute of Mathematics, 4 Acad. Koptyug avenue, 630090, Novosibirsk, Russia.}
\address{Novosibirsk State  University, 2 Pirogova Street, 630090, Novosibirsk, Russia.}
\address{Novosibirsk State Agrarian University, Dobrolyubova street, 160, Novosibirsk, 630039, Russia.}
\email{bardakov@math.nsc.ru}

\address{Department of Mathematical Sciences, Indian Institute of Science Education and Research (IISER) Mohali, Sector 81,  S. A. S. Nagar, P. O. Manauli, Punjab 140306, India.}
\email{mahender@iisermohali.ac.in}
\email{manpreetsingh@iisermohali.ac.in}


\begin{abstract}
In this note, residual finiteness of quandles is defined and investigated. It is proved that free quandles and knot quandles of tame knots are residually finite and Hopfian. Residual finiteness of quandles arising from residually finite groups (conjugation, core and Alexander quandles) is established. Further, residual finiteness of automorphism groups of some residually finite quandles is also discussed.
\end{abstract}

\maketitle

\section{Introduction}
In the recent years, quandles have been a subject of intensive investigation due to their appearance in various areas of mathematics. These objects first appeared in the work of Joyce \cite{Joyce} under the name  quandle, and that of Matveev \cite{Matveev} under the name distributive groupoid. A quandle is a set with a binary operation that satisfies three axioms modelled on the three Reidemeister moves of diagrams of knots in $\mathbb{S}^3$. Joyce and Matveev independently proved that each oriented diagram $D(K)$ of a tame knot $K$ (in fact, tame link) gives rise to a quandle $Q(K)$, called the knot quandle, which is independent of the diagram  $D(K)$. Further, they showed that if $K_1$ and $K_2$ are two tame knots with $Q(K_1) \cong Q(K_2)$, then there is a homeomorphism of $\mathbb{S}^3$ mapping $K_1$ onto $K_2$, not necessarily preserving the orientations. We refer the reader to the survey articles \cite{Carter, Kamada, Nelson} for more on the historical development of the subject and its relationships with other areas of mathematics.
\vspace*{.6mm}

Although the knot quandle is a strong invariant for tame knots, it is usually difficult to check whether two knot quandles are isomorphic. This motivates search for newer properties of quandles, in particular, of knot quandles. Over the years, various ideas have been transferred from other algebraic theories to that of quandles and their analogues (racks, shelves, etc).
\vspace*{.6mm}

The notion of residual finiteness (and other residual properties) of groups plays a crucial role in combinatorial group theory and low dimensional topology. In this note, we define and investigate residual finiteness of quandles. We begin by proving some closure properties of residual finiteness of quandles. We then investigate residual finiteness of conjugation, core, Alexander quandles of residually finite groups. Further, we discuss residual finiteness of automorphism groups of some residually finite quandles. Our first main result is that free quandles are residually finite, and finitely generated residually finite quandles are Hopfian. Our next main result is that the knot quandles of tame knots are residually finite. The key idea in its proof is the notion of finite separability of a subgroup of a group, and a result of Long and Niblo  \cite{Long-Niblo} on finite separability of  $\pi_1(X,p)$ in $\pi_1(M,p)$, where $M$ is an orientable irreducible compact 3-manifold and $X$ an incompressible connected subsurface of a component of the boundary $\partial(M)$ of $M$ containing the base point $p$. As a consequence, we obtain that knot quandles of tame knots are Hopfian.
\vspace*{.6mm}

\section{Preliminaries on quandles}\label{prelim}
We begin with the definition of the main object of our study, namely, a quandle.
\vspace*{.6mm}

\begin{definition}
A {\it quandle} is a non-empty set $X$ with a binary operation $(x,y) \mapsto x * y$ satisfying the following axioms:
\begin{enumerate}
\item $x*x=x$ for all $x \in X$;
\item For any $x,y \in X$ there exists a unique $z \in X$ such that $x=z*y$;
\item $(x*y)*z=(x*z) * (y*z)$ for all $x,y,z \in X$.
\end{enumerate}
\end{definition}
\vspace*{.6mm}

A non-empty set with a binary operation satisfying only the axioms (2) and (3) is called a {\it rack}. Obviously, every quandle is a rack, but not conversely.
\vspace*{.6mm}

\begin{example}
Although tame knots are rich sources of quandles, many interesting examples of quandles come from groups.
\begin{itemize}
\item If $G$ is a group, then the set $G$ equipped with the binary operation $a*b= b^{-1} a b$ gives a quandle structure on $G$, called the {\it conjugation quandle}, and denoted by $\Conj(G)$.
\item If $G$ is a group and we take the binary operation $a*b= b a^{-1} b$, then we get the {\it core quandle}, denoted as $\Core(G)$. In particular, if $G$ is additive abelian, then $\Core(G)$ is the Takasaki quandle of $G$.
\item Let $G$ be a group and $\varphi \in \Aut(G)$, then the set $G$ with binary operation $a*b=\varphi(ab^{-1})b$ gives a quandle structure on $G$, which is denoted by $\Alex(G, \varphi)$. These quandles are called as {\it generalized Alexander quandles}.
\end{itemize}
\end{example}
\vspace*{.6mm}

The quandle axioms are equivalent to saying that for each $x \in X$, the map $S_x: X \to X$ given by $S_x(y)=y*x$ is an automorphism of the quandle $X$ fixing $x$, called an {\it inner automorphism} of $X$. The group generated by all such automorphisms is denoted by $\Inn(X)$.  The fact that $S_x$ is a bijection for each $x \
\in X$ is equivalent to existence of another binary operation on $X$, written $(x, y) \mapsto x \ast^{-1} y$, and satisfying $$x \ast y = z~\textrm{if and only if}~ x = z \ast^{-1} y$$ for all $x, y, z \in X$. Further, the map $S: X \to \Conj \big(\Inn(X) \big)$ given by $S(x)=S_x$ is a quandle anti-homomorphism. In other words,
\begin{equation}\label{inn-anti-homo}
S_{x\ast y}= S_x \ast^{-1} S_y=  S_y \circ S_x \circ S_y^{-1}
\end{equation}
for $x, y \in X$. It is easy to see that $X$ and $Y$ are quandles and $f: X \to Y$ a map, then $f(x_1 \ast x_2)= f(x_1) \ast f(x_2)$ if and only if  $f(x_1 \ast^{-1} x_2)= f(x_1) \ast^{-1} f(x_2)$ for all $x_1, x_2 \in X$.
\vspace*{.6mm}

A quandle $X$ is called {\it trivial} if $x*y=x$ for all $x, y \in X$. Note that a trivial quandle can contain arbitrary number of elements.
\vspace*{.6mm}


\section{Residually finite quandles and some properties}\label{sec-rf-closure}
Recall that a group $G$ is called {\it residually finite} if for each $g \in G$   with $g \neq 1$, there exists a finite group $F$ and a homomorphism $\phi : G \rightarrow F $ such that $\phi(g) \neq 1$.
\vspace*{.6mm}

It is easy to see that a group $G$ being residually finite is equivalent to saying that for $g, h \in G$ with $g \neq h$, there exists a finite group $F$ and a homomorphism $\phi: G \rightarrow F $ such that $ \phi(g)\neq \phi(h)$.
\vspace*{.6mm}

The preceding observation motivates the definition of residually finite quandles.
\vspace*{.6mm}

\begin{definition}
 A quandle $X$ is said to be {\it residually finite} if for all $x,y \in X$ with $x \neq y$, there exists a finite quandle $F$ and quandle homomorphism $\phi:X \rightarrow F$ such that $\phi(x) \neq \phi(y)$.
\end{definition}
\vspace*{.6mm}

In \cite{Mal'cev}, Mal'cev gave the definition of a residually finite algebra, and proved that for some algebras residual finiteness implies that the word problem is solvable. The preceding definition is a particular case of Mal'cev's definition.
\vspace*{.6mm}

Obviously, every finite quandle is residually finite, and every subquandle of a residually finite quandle is residually finite. Further, we have the following.
\vspace*{.6mm}

\begin{proposition}\label{trivial-res-finite}
Every trivial quandle is residually finite.
\end{proposition}

\begin{proof}
Let $X$ be a trivial quandle. If $X$ has only one element, then there is nothing to prove. Suppose that $X$ has at least two elements. Let $x, y \in X$ with $x \neq y$. Consider the trivial subquandle $\{x, y \}$ of $X$ and define $\phi: X \to \{x, y \}$ by $\phi(x)=x$ and $\phi(z)=y$ for all $z \neq x$. Then it is easy to see that $\phi$ is a quandle homomorphism with $\phi(x) \neq \phi(y)$, and hence $X$ is residually finite.
\end{proof}
\vspace*{.6mm}

Next, we investigate some closure properties of residually finite quandles. Let $\lbrace X_i\rbrace_{i \in I}$ be an indexed family of quandles and $X= \prod_{i \in I}X_i$ their cartesian product. Then $X$ is itself a quandle, called {\it product quandle}, with binary operation given by  $$(x_i) \ast (y_i)= (x_i \ast y_i)$$ for $(x_i), (y_i) \in X$.  Further, for each $j \in I$, the projection map $$\pi_j: X \rightarrow X_j$$ given by $\pi_j \big((x_i)\big)=x_j$ is a quandle homomorphism.
\vspace*{.6mm}

\begin{proposition}\label{direct-prod-rf}
Let $\{X_i\}_{i \in I}$ be an indexed family of residually finite quandles. Then the product quandle $X = \prod_{i \in I}X_i$ is residually finite.
\end{proposition}

\begin{proof}
Let $x=(x_i),  y=(y_i) \in X$ such that $x \neq y$. Then there exists an $i_0 \in I$ such that $x_{i_0} \neq y_{i_0} $. Since $X_{i_0}$ is residually finite, there exists a finite quandle $F$ and a homomorphism $ \phi: X_{i_0} \rightarrow F$ such that $ \phi(x_{i_0}) \neq \phi(y_{i_0})$. The homomorphism $ \phi ' := \phi \circ \pi_{i_0}$ satisfy $ \phi '(x) \neq \phi ' (y)$, and hence $X$ is a residually finite quandle.
\end{proof}
\vspace*{.6mm}

\begin{proposition}
The following statements are equivalent for a quandle $X$:
\begin{enumerate}
\item $X$ is residually finite;
\item there exists a family $\{W_i\}_{i \in I}$ of finite quandles such that the quandle $X$ is isomorphic to a subquandle of the product quandle $\prod _{i \in I} W_i $.
\end{enumerate}
\end{proposition}

\begin{proof}
The implication $(2) \implies (1)$ follows from Proposition  \ref{direct-prod-rf} and the fact that a  subquandle of a residually finite quandle is residually finite. Conversely, suppose that $X$ is residually finite. For each pair $(x,y) \in X \times X$ such that $x \neq y$, there exists a finite quandle $W_{(x,y)}$ and a homomorphism $ \phi_{(x,y)} : X \rightarrow W_{(x,y)}$ such that $ \phi_{(x,y)}(x) \neq \phi_{(x,y)}(y)$. Now consider the quandle $$ W= \underset{(x,y)\in X \times X,~ x \neq y}{\mathrm{\prod}} W_{(x,y)},$$ and define a homomorphism $\psi : X \rightarrow W $ by $$ \psi = \underset{(x,y)\in X \times X,~ x \neq y}{\mathrm{\prod}} \phi _{(x,y)},$$ which is clearly injective. Hence $X$ is residually finite being isomorphic to a subquandle of $W$.
\end{proof}
\vspace*{.6mm}

An inverse system of quandles $\{X_i, \pi_{ij},I \}$ consists of a directed set $I$, a family of quandles $\{X_i\}_{i \in I}$, and a collection of quandle homomorphisms $\pi_{ij}:X_j \to X_i$ for $i \le j$ in $I$ satisfying the following conditions:
\begin{enumerate}
\item $\pi_{ii}= \id_{X_i}$ for each $i \in I$;
\item $\pi_{ij} \circ \pi_{jk}= \pi_{ik}$ for all $i \le j \le k$ in $I$.
\end{enumerate}
Given an inverse system $\{X_i, \pi_{ij},I \}$ of quandles, as discussed above, we construct the product quandle $X = \prod_{i \in I}X_i$. Let $\varprojlim X_i$ be the subset of $X$ consisting of elements $(x_i) \in X$ such that $x_i=\pi_{ij}(x_j)$ for $i \le j$ in $I$. It is easy to see that $\varprojlim X_i$ is subquandle of $X$ called the inverse limit of the inverse system $\{X_i, \pi_{ij},I \}$. In view of Proposition \ref{direct-prod-rf} and the fact that  every subquandle of a residually finite quandle is residually finite, we obtain the following:

\begin{corollary}
The inverse limit of an inverse system of residually finite quandles is residually finite.
\end{corollary}
\vspace*{.6mm}


\section{Residual finiteness of quandles arising from groups}\label{sec-rf-other-quandles}

In this section, we investigate residual finiteness of conjugation, core and Alexander quandles of residually finite groups. We also discuss residual finiteness of certain automorphism groups of residually finite quandles.
\vspace*{.6mm}

\begin{proposition}\label{conj-g-res-finite}
If $G$ is a residually finite group, then $\Conj(G)$ and $\Core(G)$ are both residually finite quandles.
\end{proposition}

\begin{proof}
If $g_1,g_2 \in G$ with $g_1\neq g_2$, then there exists a finite group $F$ and a group homomorphism $\phi :G \rightarrow F $ such that $\phi(g_1) \neq \phi(g_2)$. The map $\Conj (\phi): \Conj(G) \rightarrow \Conj(F) $ given by $ \Conj (\phi)(g)=\phi(g)$ for $g \in \Conj(G)$ is a quandle homomorphism with $\Conj (\phi)(g_1) \neq \Conj (\phi)(g_2)$.  Similarly, $\Core(G)$ is residually finite.
\end{proof}
\vspace*{.6mm}

For generalised Alexander quandles, we prove
\vspace*{.6mm}

\begin{proposition}
Let $G$ be a residually finite group. If $\alpha :G \rightarrow G$ is an inner automorphism, then $\Alex(G, \alpha)$ is a  residually finite quandle.
\end{proposition}

\begin{proof}
Let $\alpha$ be the inner automorphism induced by $g_0 \in G$. If $g_1, g_2 \in G$ such that $g_1 \neq g_2$, then there exists a finite group $F$ and a group homomorphism $\psi:G \rightarrow F$ such that $\psi(g_1) \neq \psi(g_2)$. Let  $\beta$ be the inner automorphism of $F$ induced by $\psi(g_0)$. It follows that $\psi$ viewed as a map $\psi: \Alex(G, \alpha) \rightarrow \Alex(F, \beta)$  is a quandle homomorphism with $\psi(g_1) \neq \psi(g_2)$, and hence $\Alex(G, \alpha)$ is residually finite.
\end{proof}
\vspace*{.6mm}

It is well-known that the automorphism group of a finitely generated residually finite group is residually finite \cite[p.414]{Magnus-Karrass-Solitar}. For the inner automorphism group of residually finite quandles, we have the following result.
\vspace*{.6mm}

\begin{theorem}
If $X$ is a residually finite quandle, then $\Inn(X)$ is a residually finite group.
\end{theorem}

\begin{proof}
If $X=\langle x_i~|~i \in I \rangle$, then $\Inn(X)=\langle S_{x_i}~|~i \in I \rangle$. Let $S_{a_1}^{e_1} \circ S_{a_2}^{e_2} \circ \cdots \circ S_{a_m}^{e_m}\neq 1$ be an element of $\Inn(X)$, where $a_j \in \{ x_i~|~i \in I\}$ and $e_j \in \{1, -1\}$. Then there exists an element $x \in X$ such that $$S_{a_1}^{e_1} \circ S_{a_2}^{e_2} \circ \cdots \circ S_{a_m}^{e_m}(x) \neq x,$$ equivalently
$$\big(\big((x \ast^{e_m} a_m) \ast^{e_{m-1}} a_{m-1} \big) \cdots \big) \ast ^{e_1} a_1 \neq x.$$ Since $X$ is residually finite, there exists a finite quandle $F$ and a quandle homomorphism $\phi: X \rightarrow F$ such that
\begin{equation}\label{phi-big}
\phi\big(\big(\big((x \ast^{e_m} a_m) \ast^{e_{m-1}} a_{m-1} \big) \cdots \big) \ast ^{e_1} a_1\big) \neq \phi(x).
\end{equation}

Define a map $$\widetilde \phi: \big\{ S_{x_i}^{\pm 1}~|~i \in I \big\} \rightarrow \Inn(F)$$ by setting $$\widetilde \phi\big(S_{x_i}^{\pm 1}\big)=S_{\phi(x_i)}^{\pm 1}.$$ We claim that $\widetilde \phi$ preserves relations in $\Inn(X)$, and hence extends to a group homomorphism. Observe that relations in $\Inn(X)$ are induced by relations in $X$. If $x\ast y=z$ is a relation in $X$, by \eqref{inn-anti-homo}, the induced relation in $\Inn(X)$ is $$S_z \circ  S_y= S_y \circ S_x.$$ Since $\phi$ is a quandle homomorphism, we  have $\phi(x) \ast \phi(y)=\phi(z)$ in $F$. Again, by \eqref{inn-anti-homo}, we have $$S_{\phi(z)} \circ  S_{\phi(y)}= S_{\phi(y)} \circ S_{\phi(x)}.$$ This proves our claim, and hence $\widetilde \phi$ extends to a group homomorphism $\widetilde \phi: \Inn(X) \to \Inn(F)$. If $\widetilde \phi \big(S_{a_1}^{e_1} \circ S_{a_2}^{e_2} \circ \cdots \circ S_{a_m}^{e_m} \big) = 1$, then evaluating both the sides at $\phi(x)$ contradicts \eqref{phi-big}. Hence, $\Inn(X)$ is a residually finite group.
\end{proof}
\vspace*{.6mm}

Next, we present some observations for automorphism groups of core and conjugation quandles of residually finite groups.
\vspace*{.6mm}

\begin{proposition}
If $G$ is a finitely generated abelian group with no $2$-torsion, then $\Aut\big(\Core(G)\big)$ is residually finite group.
\end{proposition}

\begin{proof}
Since $G$ is a finitely generated abelian group, it is residually finite, and hence $\Aut(G)$ is also residually finite. Moreover, semi-direct product of residually finite groups is residually finite. By \cite[Theorem 4.2]{Bardakov-Dey-Singh}, $\Aut\big(\Core(G)\big) \cong G \rtimes \Aut(G)$, and hence $\Aut\big(\Core(G)\big)$ is residually finite.
\end{proof}
\vspace*{.6mm}

\begin{proposition}
If $G$ is a finitely generated residually finite group with trivial centre, then $\Aut \big(\Conj(G)\big)$ is residually finite.
\end{proposition}

\begin{proof}
Since $G$ has trivial center, by \cite[Corollary 4.2]{Bardakov-Nasybullov-Singh},  $\Aut \big(\Conj(G)\big)=\Aut(G)$, which is residually finite as $G$ is so.
\end{proof}
\vspace*{.6mm}

\section{Residual finiteness of free quandles}\label{sec-rf-free-quandles}

In this section, we consider residual finiteness of free quandles, the free objects in the category of quandles.
\vspace*{.6mm}

\begin{definition}
A free quandle on a non-empty set $S$ is a quandle $FQ(S)$ together with a map $\phi:S \rightarrow FQ(S)$ such that for any other map $\rho:S \rightarrow X$, where $X$ is a quandle, there exists a unique quandle homomorphism $\bar{\mathbb{\rho}}:FQ(S) \rightarrow X$ such that $\bar{\mathbb{\rho}} \circ \phi =\rho$.
\end{definition}
\vspace*{.6mm}

A free rack is defined analogously. It follows from the definition that a free quandle (a free rack) is unique up to isomorphism, and every quandle (rack) is quotient of a free quandle (rack).
\vspace*{.6mm}

The following construction of a free rack is due to Fenn and Rourke \cite[p.351]{Fenn-Rourke}. Let $S$ be a set and $F(S)$ the free group on $S$. Define
 $$FR(S):=S \times F(S) = \big\lbrace a^{w}:= (a,w) \mid a \in S, w \in F(S) \big\rbrace $$ with the operation defined as $$a^{w} \ast b^{u}:= a^{wu^{-1}bu}.$$
It can be seen that $FR(S)$ is a free rack on $S$. 

A model of the free quandle on the set $S$ is due to Kamada \cite{Kamada2012, Kamada2017}, who defined the free quandle $FQ(S)$ on $S$ as a quotient of $FR(S)$ modulo the equivalence relation generated by $$a^{w}= a^{aw}$$ for $a \in S$ and $w \in F(S)$. It is not difficult to check that $FQ(S)$ is quandle satisfying the above universal property.
\par

There is another model of free quandle on a set $S$ \cite[Example 2.16]{Nosaka}, which is defined as the subquandle of $\Conj\big(F(S) \big)$ consisting of all conjugates of elements of $S$. For the benefit of readers, we present an explicit isomorphism between the two models.

\begin{proposition}\label{equivalence-two-models}
The map $\Phi: FQ(S) \rightarrow \Conj \big(F(S)\big)$ given by $\Phi (a^{w})=w^{-1} aw$ is an embedding of quandles.
 \end{proposition}

\begin{proof}
Let $a_{1}^{w_{1}},a_{2}^{w_{2}} \in FQ(S)$. Then $\Phi(a_{1}^{w_{1}})=w_{1}^{-1}a_{1}w_{1}$, $\Phi(a_{2}^{w_{2}})=w_{2}^{-1}a_{2}w_{2}$ and $a_{1}^{w_{1}} \ast a_{2}^{w_{2}}= a_{1}^{w_{1}w_{2}^{-1}a_{2}w_{2}}$. Further,
\begin{align*}
\Phi(a_{1}^{w_{1}} \ast a_{2}^{w_{2}})&=\Phi(a_{1}^{w_{1}w_{2}^{-1}a_{2}w_{2}})\\
&=(w_1w_2^{-1}a_2w_2)^{-1}a_1(w_1w_2^{-1}a_2w_2)\\
&=w_2^{-1}a_2^{-1}w_2w_{1}^{-1}a_1w_1w_2^{-1}a_2w_2\\
&=(w_2^{-1}a_2w_2)^{-1}(w_1^{-1}a_1w_1)(w_2^{-1}a_2w_2)\\
&=\Phi(a_1^{w_1}) \ast \Phi(a_2^{w_2}),
\end{align*}
and hence $\Phi$ is a quandle homomorphism.  Let $a_1^{w_1},  a_2^{w_2} \in FQ(S)$ such that $a_1^{w_1}$  $ \neq $  $a_2^{w_2}$.
\par
Case 1: Suppose $ a_1 $ $ \neq $ $ a_2 $. If $\Phi( a_1^{w_1} )= \Phi( a_2^{w_2} )$, then $ w_1^{-1} a_1 w_1 = w_2^{-1} a_2 w_2$, which contradicts the fact that $F(S)$ is a free group. Hence  $\Phi( a_1^{w_1} ) \neq \Phi( a_2^{w_2} )$.
\par
Case 2: Suppose $a_1 = a_2=a$. If $\Phi( a^{w_1} )= \Phi( a^{w_2} )$, then $ w_1^{-1} a w_1 = w_2^{-1} a w_2$, which further implies that $ w_1 w_2 ^{-1} $ commutes with $a$ in $F(S)$. Since $F(S)$ is a free group, only powers of $a$ can commute with $a$, and hence $w_1 w_2^{-1} $ = $a^{i}$ for some integer $i$. Thus $w_1$= $a^{i} w_2$, which implies that $ a^{w_1}$ = $ a^{a^{i} w_2} = a^{ w_2}$ in $ FQ(S)$,  a contradiction. Hence $\Phi( a^{w_1} ) \neq \Phi( a^{w_2} )$, and $\Phi$ is an embedding of quandles.
\end{proof}
\vspace*{.6mm}

\begin{theorem}\label{free-quandle-rf}
Every free quandle is residually finite.
\end{theorem}

\begin{proof}
Let $FQ(S)$ be the free quandle on the set $S$. It is well-known that the free group $F(S)$ is residually finite \cite[Theorem 2.3.1]{Silberstein-Coornaert}. By Proposition \ref{conj-g-res-finite}, the quandle $\Conj\big(F(S)\big)$ is residually finite. Since $FQ(S)$ is a subquandle of $\Conj \big(F(S)\big)$, it follows that $FQ(S)$ is residually finite.
\end{proof}
\vspace*{.6mm}

The following is a well-known result for free groups \cite[p.42]{Kurosh}.
\vspace*{.6mm}

\begin{theorem}\label{free-group-symm}
If $F(S)$ is a free group on a set $S$ and $g \neq 1$ an element of $F(S)$, then there is a homomorphism $\rho: F(S) \to \Sa_n$ for some $n$ such that $\rho(g) \neq 1$, where $\Sa_n$ is the symmetric group on $n$ elements.
\end{theorem}
\vspace*{.6mm}

We prove an analogue of the preceding result for free quandles.
\vspace*{.6mm}

\begin{theorem}
Let $FQ(S)$ be a free quandle on a set $S$ and $x,y \in FQ(S)$ such that $x \neq y$. Then there is a quandle homomorphism $\phi:FQ(S) \rightarrow \Conj(\Sa_{n})$ for some $n$ such that $\phi(x) \neq \phi(y)$.
\end{theorem}

\begin{proof}
Recall that the map $\Phi: FQ(S) \rightarrow \Conj \big(F(S)\big) $ in Theorem \ref{free-quandle-rf} is an injective quandle homomorphism. Let $a_{1}^{w_{1}} \neq a_{2}^{w_{2}} \in FQ(S)$. Then $g_{1} \neq g_{2} \in F(S)$, where $g_{1}=\Phi(a_{1}^{w_{1}})$ and $g_{2}=\Phi(a_{2}^{w_{2}})$. Thus, $g_{2}^{-1}g_{1}$ is a non-trivial element of $F(S)$. By Theorem \ref{free-group-symm}, there exists a symmetric group $\Sa_{n}$ for some $n$ and a group homomorphism $\rho:F(S) \rightarrow \Sa_{n}$ such that $\rho(g_{1})\neq\rho(g_{2})$. Let $\Conj(\rho):\Conj \big(F(S)\big) \rightarrow \Conj(\Sa_{n})$ be the induced map with $\Conj(\rho)(g_{1})\neq \Conj(\rho)(g_{2})$. Taking  $\phi := \Conj(\rho) \circ \Phi : FQ(S) \rightarrow \Conj(\Sa_{n})$, we see that $\phi (a_{1}^{w_{1}}) \neq \phi(a_{2}^{w_{2}})$.
\end{proof}
\vspace*{.6mm}

\begin{definition}
A quandle $X$ is called {\it Hopfian} if every surjective quandle endomorphism of $X$ is injective.
\end{definition}
\vspace*{.6mm}

It is well-known that finitely generated residually finite groups are Hopfian  \cite{Mal'cev1}. We prove a similar result for quandles.
\vspace*{.6mm}

\begin{theorem}\label{hopfian-thm}
Every finitely generated residually finite quandle is Hopfian.
\end{theorem}

\begin{proof}
Let $X$ be a finitely generated residually finite quandle and $\phi:X \rightarrow X$ a surjective quandle homomorphism. Suppose that $\phi$ is not injective. Let $x_1, x_2 \in X$ such that $ x_1 \neq x_2$ and $\phi(x_1) = \phi(x_2)$. Since $X$ is residually finite, there exist a finite quandle $F$ and a quandle homomorphism $\tau:X \rightarrow F$ such that $\tau(x_1) \neq \tau(x_2)$.
\par
We claim that the maps $\tau \circ \phi ^{n}:X \rightarrow F$ are distinct quandle homomorphisms for all $n \geq 0$. Let $0 \leq m <n$ be integers. Since $$ \phi ^{m}:X \rightarrow X$$ is surjective, there exist $y_1, y_2 \in X$ such that $\phi ^{m}(y_1)=x_1$ and $\phi ^{m}(y_2)=x_2$. Thus, we have $$\tau \circ \phi ^{m}(y_1) \neq \tau \circ \phi ^{m}(y_2),$$ whereas $$\tau \circ \phi ^{n}(y_1) = \tau \circ \phi ^{n}(y_2),$$ which proves our claim. Thus, there are infinitely many quandle homomorphisms from $X$ to $F$, which is a contradiction, since $X$ is finitely generated and $F$ is finite. Hence, $\phi$ is an automorphism, and  $X$ is Hopfian.
\end{proof}
\vspace*{.6mm}

By theorems \ref{free-quandle-rf} and \ref{hopfian-thm}, we obtain

\begin{corollary}\label{fg-free-quandle-hopf}
Every finitely generated free quandle is Hopfian.
\end{corollary}
\vspace*{.6mm}

\begin{remark}
The preceding result is not true for infinitely generated free quandles (free racks). Indeed, if $FQ_{\infty}$ is a free quandle that is freely generated by an infinite set $\{x_1, x_2, \ldots \}$, then we can define a homomorphism $\varphi : FQ_{\infty} \to FQ_{\infty}$ by setting
$$
\varphi(x_1) = x_1~\textrm{and}~\varphi(x_i) = x_{i-1}
$$
for $i\ge 2$. It is easy to see that $\varphi$ is an epimorphism which is not an automorphism since $\varphi(x_1) = \varphi(x_2)$.
\end{remark}
\vspace*{.6mm}

The \textit{enveloping group} of a quandle $Q$, denoted by $G_Q$, is the group with $Q$ as the set of generators and defining relations $$x \ast y=y^{-1} x y$$ for all $x,y\in Q$. For example, if $Q$ is a trivial quandle, then $G_Q$ is the free abelian group of rank the cardinality of $Q$.
\vspace*{.6mm}

Since every quandle is quotient of a free quandle, a quandle $Q$ can be defined by a set of generators and defining relations as
$$
Q = \big\langle X~||~R \big\rangle.
$$
For example, knot (link) quandles have such presentations.
\vspace*{.6mm}

\begin{proposition}\label{rank-of-free-quandle}
Let $FQ(S)$ and $ FQ(T)$ be free quandles on sets $S$ and $T$, respectively. If $FQ(S)\cong FQ(T)$, then $|S|=|T|$.
\end{proposition}

\begin{proof}
By \cite[p.106, Theorem 5.1.7]{Winker}, if $Q$ is a quandle with a presentation $Q=  \langle X~||~R \rangle$, then its enveloping group has presentation $G_Q \cong  \langle X~||~\bar{R} \rangle$, where $\bar{R}$ consists of relations in $R$ with each expression $x \ast y$ replaced by $y^{-1} x y$. Consequently, since $FQ(S)$ and $FQ(T)$ are free quandles, it follows that $G_{FQ(S)}=F(S)$ and $G_{FQ(T)}=F(T)$ are free groups on the sets $S$ and $T$, respectively. Since $FQ(S)\cong FQ(T)$, we must have $G_{FQ(S)} \cong G_{FQ(T)}$, and hence $|S|=|T|$.
\end{proof}
\vspace*{.6mm}

In view of Proposition \ref{rank-of-free-quandle}, we can define the {\it rank of a free quandle} as the cardinality of its any free generating set.
\vspace*{.6mm}

Analogous to groups, we define the {\it word problem} for quandles as the problem of determining whether two given elements of a quandle are the same. The word problem is solvable for finitely presented residually finite groups  \cite[p.55, 2.2.5]{Robinson}. Below is a similar result for quandles.
\vspace*{.6mm}

\begin{theorem}\label{word-problem}
Every finitely presented residually finite quandle has a solvable word problem.
\end{theorem}

\begin{proof}
Let $Q=\langle  X~||~R\rangle $ be a  finitely presented residually finite quandle, and  $w_1,w_2$ two words in the generators $X$. We describe two procedures which tell us whether or not $w_1 =w_2$ in $Q$. The first procedure lists all the words  that we obtain by using  the relations of $Q$ on the word $w_1$. If the word $w_2$ turns up at some stage, then $w_1=w_2$, and we are done.
\par
The second procedure lists all the finite quandles. Since $Q$ is finitely generated, for each finite quandle $F$, the set $\Hom(Q,F)$ of all quandle homomorphisms is finite. Now for each homomorphism $\phi \in \Hom(Q,F)$, we look for $\phi(w_1)$ and $\phi(w_2)$ in $F$, and check whether or not $\phi(w_1)=\phi(w_2)$. Since $Q$ is  residually finite, the above procedure must stop at some time. That is, there exists a finite quandle $F$ and  $\phi \in \Hom(Q,F)$  such that  $\phi(w_1) \neq \phi(w_2)$ in $F$, and hence $w_1 \neq w_2$ in $Q$.
\end{proof}
\vspace*{.6mm}

\begin{remark}
In a recent work \cite{Belk-McGrail}, Belk and McGrail showed that the word problem for quandles is unsolvable in general by giving an example of a finitely presented quandle with unsolvable word problem. In view of Theorem \ref{word-problem}, such a quandle cannot be residually finite.
\end{remark}
\vspace*{.6mm}

\section{Residual finiteness of knot quandles}\label{sec-rf-knot-quandles}

In this section, we prove that the knot quandle of a tame knot is residually finite. We recall the following definition from \cite{Mal'cev}.
\vspace*{.6mm}

\begin{definition}
A subgroup $H$ of a group $G$ is said to be {\it finitely separable} in $G$ if for each $g \in G\setminus H$, there exists a finite group $F$ and a group homomorphism $\phi: G \to F$ such that $\phi(g) \not\in \phi(H)$.
\end{definition}
\vspace*{.6mm}

Let $H$ be a subgroup of a group $G$ and  $G/H$ the set of right cosets of $H$ in $G$. For $g \in G$, we denote its right coset by $\bar{g}$.  Let $z \in \C_G(H)$, the centraliser of $H$ in $G$, be a fixed element. Then it is easy to see that the set $G/H$ with the binary operation given by $$\bar{x} \ast \bar{y} = \bar{z}^{-1} \bar{x} \bar{y}^{-1} \bar{z} \bar{y}$$ for $\bar{x}, \bar{y} \in G/H$ forms a quandle, denoted $(G/H,z)$.
\vspace*{.6mm}

\begin{proposition}\label{res-finite-quandle}
Let $H$ be a subgroup of a group $G$ and $z\in  \C_G(H)$. If $H$ is finitely separable in $G$, then the quandle $(G/H,z)$ is residually finite.
\end{proposition}

\begin{proof}
Let $\bar g_1,\bar g_2 \in G/H$ such that $\bar g_1 \neq \bar g_2$, that is, $g_{1} \neq hg_{2}$ for any $h \in H$. Since $H$ is finitely separable in $G$, there exists a finite group $F$ and a group homomorphism $\phi:G \rightarrow F$ such that $\phi(g_{1}) \neq \phi(hg_{2})$ for each $h \in H$. Let $\overline{H}:= \phi(H)$ and $\bar{z}:=\phi(z) \in \C_F(\overline{H})$. Then $(F/\overline{H},\bar{z})$ is a finite quandle. Further, the group homomorphism $\phi:G \rightarrow F$ induces a well-defined map $$\bar{\mathbb{\phi}}:(G/H,z) \rightarrow (F/\overline{H},\bar{z})$$ given by $$\bar{\mathbb{\phi}}( \bar x)=\overline{H}\phi(x),$$ which is a quandle homomorphism. Also, $\bar{\mathbb{\phi}}(\bar g_1) \neq \bar{\mathbb{\phi}}(\bar g_2)$, otherwise $\phi(g_{1})=\phi(hg_{2})$ for some $h\in H$, which is a contradiction. Hence the quandle $(G/H,z)$ is residually finite.
\end{proof}
\vspace*{.6mm}

\begin{definition}
A subquandle $Y$ of a quandle $X$ is said to be {\it finitely separable} in $X$ if for each $x \in X\setminus Y$, there exists a finite quandle $F$ and a quandle homomorphism $\phi: X \to F$ such that $\phi(x) \not\in \phi(Y)$.
\end{definition}
\vspace*{.6mm}

The following result might be of independent interest.
\vspace*{.6mm}

\begin{proposition}
Let $X$ be a residually finite quandle and $\alpha \in \Aut(X)$. If $\Fix(\alpha):= \{x \in X~|~ \alpha(x)=x\}$ is non-empty, then it is a finitely separable subquandle of $X$.
\end{proposition}

\begin{proof}
Clearly $\Fix(\alpha)$ is a subquandle of $X$. Let $x_0 \in X \setminus \Fix(\alpha)$, that is, $\alpha(x_0) \neq x_0$. Since $X$ is residually finite, there exists a finite quandle $F$ and a quandle homomorphism $\phi: X \rightarrow F$ such that $\phi(\alpha(x_0)) \neq \phi(x_0)$. Define a map $\eta: X \rightarrow F\times F $ by $\eta(x)=\big(\phi(x), \phi(\alpha(x))\big)$. Clearly $\eta$ is a quandle homomorphism with $\eta(x_0) \not\in \eta(X)$, and hence $\Fix(\alpha)$ is finitely separable in $X$.
\end{proof}
\vspace*{.6mm}

Answering a question raised by Jaco \cite[V.22]{Jaco}, Long and Niblo \cite{Long-Niblo} proved the following result using the fact that doubling a 3-manifold along its boundary preserves residual finiteness.
\vspace*{.6mm}

\begin{theorem}\label{Long-Niblo-thm}
Suppose that $M$ is an orientable, irreducible compact 3-manifold and $X$ an incompressible connected subsurface of a component of $\partial(M)$. If $p\in X$ is a base point, then $\pi_1(X,p)$ is a finitely separable subgroup of $\pi_1(M,p)$.
\end{theorem}
\vspace*{.6mm}

A group $G$ is said to be {\it subgroup separable} if every finitely generated subgroup of $G$ is finitely separable in $G$. It is well-known that not all 3-manifold groups are subgroup separable (see \cite{Long-Niblo}). We refer the reader to \cite{Scott} for relation between subgroup separability and geometric topology.
\vspace*{.6mm}

Let $V(K)$ be a tubular neighbourhood of a knot $K$ in $\mathbb{S}^3$. Then the knot complement $C(K):=\overline{\mathbb{S}^3 \setminus V(K)}$ has boundary $\partial C(K)$ a torus. Let $x_0 \in \partial C(K)$ a fixed base point. Then the inclusion
$$\iota:\partial C(K) \longrightarrow  C(K)$$ induces a group homomorphism $$\iota_*:\pi_1\big(\partial C(K),x_0\big) \longrightarrow \pi_1\big(C(K),x_0\big),$$ which is injective unless the knot $K$ is trivial \cite[p.41, Proposition 3.17]{Burde-Zieschang}. In fact, $\pi_1\big(\partial C(K),x_0\big) \cong \mathbb{Z} \oplus \mathbb{Z}$, and if $K$ is trivial, then $\pi_1\big(C(K),x_0\big) \cong \mathbb{Z}$. The group $P:= \iota_*\big(\pi_1(\partial C(K),x_0)\big)$ is called the {\it peripheral subgroup} of the knot group $\pi_1\big(C(K),x_0\big)$. Now, an immediate consequence of Theorem \ref{Long-Niblo-thm} is the following result.
\vspace*{.6mm}

\begin{corollary}\label{per-sep-group}
The peripheral subgroup of a non-trivial tame knot is finitely separable in the knot group.
\end{corollary}
\vspace*{.6mm}

We need the following result of Joyce  \cite[Section 4.9]{Joyce-thesis}, which follows by observing that the knot group $G$ of a tame knot $K$ acts transitively on its knot quandle $Q(K)$ with the stabiliser of an element of $Q(K)$ being isomorphic to the peripheral subgroup $P$.
\vspace*{.6mm}

\begin{proposition}\label{joyce-theorem}
Let $K$ be a tame knot with knot group $G$ and knot quandle $Q(K)$. Let $P$ be the peripheral subgroup of $G$ containing the meridian $m$. Then the knot quandle $Q(K)$ is isomorphic to the quandle $(G/P, m)$.
\end{proposition}
\vspace*{.6mm}

We now have our main result.
\vspace*{.6mm}

\begin{theorem}\label{knot-quandle-rf}
The knot quandle of a tame knot is residually finite.
\end{theorem}

\begin{proof}
Let $K$ be a tame knot. If $K$ is an unknot, then the knot quandle $Q(K)$ is vacuously residually finite being a trivial quandle with one element. If $K$ is non-trivial, then using Proposition \ref{joyce-theorem}, Corollary \ref{per-sep-group} and Proposition \ref{res-finite-quandle} it follows that $Q(K)$ is residually finite.
\end{proof}
\vspace*{.6mm}
As a consequence of Theorem \ref{word-problem} and \ref{knot-quandle-rf}, it follows that the word problem is solvable in knot quandles of tame knots.
\vspace*{.6mm}

Theorems \ref{hopfian-thm} and \ref{knot-quandle-rf} yield the following.
\vspace*{.6mm}

\begin{corollary}
The knot quandle of a tame knot is Hopfian.
\end{corollary}
\vspace*{.6mm}

An immediate consequence of Theorem \ref{knot-quandle-rf} is the following

\begin{corollary}
Let $K$ be a non-trivial tame knot. Then there exists a finite quandle $X$ such that $\Hom\big(Q(K), X\big)$ has a non-constant homomorphism.
\end{corollary}
\vspace*{.6mm}

\begin{remark}
Joyce \cite[pp. 47--48]{Joyce} used a result of Waldhausen \cite{Waldhausen} to prove that quandles associated to tame knots are complete invariants up to orientation. To use \cite{Waldhausen} the knot complements are required to be irreducible 3-manifolds. But, this is not always true for tame links since there are tame links whose complements in $\mathbb{S}^3$ are reducible 3-manifolds. Thus, quandles associated to tame links are not complete invariants. For the same reason, Theorem \ref{Long-Niblo-thm} is not applicable, and hence we are not able to prove an analogue of Theorem \ref{knot-quandle-rf} for tame links with more than one component.
\end{remark}
\vspace*{.6mm}

\begin{problem}
We conclude with the following problems which might shed more light on the ideas pursued in this paper.
\begin{enumerate}
\item Let $L$ be a tame link with more than one component. Is it true that the link quandle $Q(L)$ is residually finite? We note that if $L_n$ is a trivial $n$-component link, then the link quandle $Q(L_n)$ is isomorphic to the free quandle on $n$ generators, and hence is residually finite by Theorem \ref{free-quandle-rf}.
\item Is it true that any subquandle of a free quandle is free?
\end{enumerate}
\end{problem}
\vspace*{.6mm}

\begin{ack}
The authors thank the anonymous referee for useful comments and for the references \cite{Fenn-Rourke,  Kamada2012, Kamada2017, Nosaka}. Bardakov acknowledges support from the Russian Science Foundation project N 16-41-02006. Mahender Singh acknowledges support from INT/RUS/RSF/P-02 grant and SERB Matrics Grant. Manpreet Singh thanks IISER Mohali for the PhD Research Fellowship.
\end{ack}


\begin{thebibliography}{HD}
\bibitem{Bardakov-Dey-Singh} V. G. Bardakov, P. Dey and M. Singh, \textit{Automorphism groups of quandles arising from groups},  Monatsh. Math. 184 (2017), 519--530.
\bibitem{Bardakov-Nasybullov-Singh} V. G. Bardakov, T. Nasybullov and M. Singh, \textit{Automorphism groups of quandles and related groups}, Monatsh. Math. (2018), https://doi.org/10.1007/s00605-018-1202-y.
\bibitem{Belk-McGrail} J. Belk and R. W. McGrail, \textit{The word problem for finitely presented quandles is undecidable}, Logic, language, information, and computation, 1--13,  Lecture Notes in Comput. Sci., 9160, Springer, Heidelberg, 2015. 
\bibitem{Burde-Zieschang} G. Burde and H. Zieschang, \textit{Knots}, Second edition. De Gruyter Studies in Mathematics, 5. Walter de Gruyter \& Co., Berlin, 2003. xii+559 pp.
\bibitem{Carter} J. Carter, \textit{A survey of quandle ideas}, Introductory lectures on knot theory, 22--53, Ser. Knots Everything, 46, World Sci. Publ., Hackensack, NJ, 2012.
\bibitem{Fenn-Rourke} R. Fenn and C. Rourke, \textit{Racks and links in codimension two}, J. Knot Theory Ramifications 1 (1992), no. 4, 343--406. 
\bibitem{Jaco} W.  Jaco, \textit{Lectures on three-manifold topology}, CBMS Regional Conference Series in Mathematics, 43. American Mathematical Society, Providence, R.I., 1980. xii+251 pp.
\bibitem{Joyce} D. Joyce, \textit{A classifying invariant of knots, the knot quandle}, J. Pure Appl. Algebra 23 (1982), 37--65.
\bibitem{Joyce-thesis} D. Joyce, \textit{An Algebraic Approach to Symmetry With Applications to Knot Theory}, PhD Thesis, University of Pennsylvania, 1979. vi+63 pp.
\bibitem{Kamada} S. Kamada, \textit{Knot invariants derived from quandles and racks}, Invariants of knots and 3-manifolds (Kyoto, 2001), 103--117, Geom. Topol. Monogr., 4, Geom. Topol. Publ., Coventry, 2002.
\bibitem{Kamada2012} S. Kamada, \textit{Kyokumen musubime riron (Surface-knot theory)}, (in Japanese), Springer Gendai Sugaku Series 16 (2012), Maruzen Publishing Co. Ltd.
\bibitem{Kamada2017} S. Kamada, \textit{Surface-knots in 4-space. An introduction}, Springer Monographs in Mathematics. Springer, Singapore, 2017. xi+212 pp.
\bibitem{Kurosh} A. G. Kurosh, \textit{The theory of groups}, Vol. II, Translated from the Russian and edited by K. A. Hirsch. Chelsea Publishing Company, New York, N.Y., 1956. 308 pp.
\bibitem{Long-Niblo} D. D. Long and G. A. Niblo, \textit{Subgroup separability and 3-manifold groups}, Math. Zeit. 207 (1991), 209--215.
\bibitem{Lyndon-Schupp} R. C. Lyndon and P. E. Schupp, \textit{Combinatorial Group Theory},  Ergebnisse der Mathematik und ihrer Grenzgebiete, Band 89. Springer-Verlag, Berlin-New York, 1977. xiv+339 pp.
\bibitem{Magnus-Karrass-Solitar}  W. Magnus, A. Karrass and D. Solitar, \textit{Combinatorial group theory: Presentations of groups in terms of generators and relations}, Interscience Publishers [John Wiley \& Sons, Inc.], New York-London-Sydney, 1966 xii+444 pp.
\bibitem{Mal'cev} A. I. Mal'cev, \textit{On homomorphism onto finite groups}, Uchen. Zap. Ivanov. Ped. Inst. 18 (1958), 49--60 (Russian).
\bibitem{Mal'cev1} A. I. Mal'cev, \textit{On isomorphic matrix representations of infinite groups}, Mat. Sb. 8 (1940), 405--422.
\bibitem{Matveev} S. Matveev, \textit{Distributive groupoids in knot theory},  Mat. Sb. (N.S.) 119:1 (1982), 78--88. In Russian; translated in Math. USSR Sb. 47:1 (1984), 73--83.
\bibitem{Nelson} S. Nelson, \textit{The combinatorial revolution in knot theory}, Notices Amer. Math. Soc. 58 (2011), 1553--1561.
\bibitem{Nosaka} T. Nosaka, \textit{Quandles and topological pairs. Symmetry, knots, and cohomology}, SpringerBriefs in Mathematics. Springer, Singapore, 2017. ix+136 pp.
\bibitem{Robinson} D. J. S. Robinson, \textit{A Course in the Theory of Groups}, Second edition. Graduate Texts in Mathematics, 80. Springer-Verlag, New York, 1996. xviii+499 pp.
\bibitem{Scott} G. P. Scott, \textit{Subgroups of surface groups are almost geometric}, J. Lond. Math. Soc. 17 (1978), 555--565.
\bibitem{Silberstein-Coornaert} T. Ceccherini-Silberstein and M. Coornaert, \textit{Cellular automata and groups}, Springer Monographs in Mathematics. Springer-Verlag, Berlin, 2010. xx+439 pp.
\bibitem{Waldhausen} F. Waldhausen, \textit{On irreducible 3-manifolds which are sufficiently large}, Ann. of Math. 87 (1968), 56--€"88.
\bibitem{Winker} S. Winker, \textit{Quandles, knots invariants and the $n$-fold branched cover}, Ph.D.Thesis, University of Illinois at Chicago (1984), 198 pp.
\end{thebibliography}
\end{document}